\documentclass{article}%
\usepackage{theorem}
\usepackage{amsmath}
\usepackage{amsfonts}
\usepackage{amssymb}
\usepackage{graphicx}
\usepackage{color}%
\providecommand{\U}[1]{\protect\rule{.1in}{.1in}}

\newtheorem{theorem}{Theorem}

\newtheorem{corollary}[theorem]{Corollary}

\newtheorem{definition}[theorem]{Definition}

\newtheorem{lemma}[theorem]{Lemma}

{\theorembodyfont{\rmfamily}
\newtheorem{remark}[theorem]{Remark}
}

\newenvironment{proof}[1][Proof]{\noindent\textbf{#1} }{\ \rule{0.5em}{0.5em}}

\newcommand*\re{\mathbb{R}}
\newcommand*\R{\mathbb{R}}
\newcommand*\N{\mathbb{N}}

\newcommand*\Omegabar{\overline{\Omega}}
\newcommand*\delomega{\partial\Omega}

\newcommand*\intomega{\int_{\Omega}}
\newcommand*\diw{\operatorname{div}}

\newcommand*\fint{\frac{1}{|\Omega|}\int}

\begin{document}

\title{The equation $\diw u+\langle a,u\rangle=f$}
\maketitle
\centerline{\scshape Pierre Bousquet$^1$ and Gyula Csat\'{o}$^2$}
\medskip 
{\footnotesize
  \centerline{1. Universit\'e de Toulouse, Toulouse, France.}

 \centerline{2. Universitat Polit\`ecnica de Catalunya, member of BGSMath Barcelona, Spain, supported }
 \centerline{by Fondecyt grant no. 11150017, and by the Mar\'ia de Maeztu Grant MDM-2014-0445.}
  
} 

\begin{abstract}
We study the solutions $u$ to the equation 
$$
\begin{cases}
\diw u + \langle a , u \rangle = f & \textrm{ in } \Omega,\\
u=0   & \textrm{ on } \partial \Omega,
\end{cases}
$$
where $a$ and $f$ are given. We significantly improve the existence results of [Csat\'o and Dacorogna, A Dirichlet problem involving the divergence operator, \textit{Ann. Inst. H. Poincar\'e Anal. Non Lin\'eaire}, 33 (2016), 829--848], where this equation has been considered for the first time. In particular, we prove the existence of a solution under essentially sharp regularity assumptions on the coefficients. The condition that we require on the vector field $a$ is  necessary and sufficient. Finally, our results cover the whole scales of Sobolev and H\"older spaces. 
\end{abstract}

\let\thefootnote\relax\footnotetext{\textit{2010 Mathematics Subject Classification.} Primary 35F15, Secondary 46N05 }
\let\thefootnote\relax\footnotetext{\textit{Key words and phrases.} divergence operator, boundary value problem, regularity.}
\let\thefootnote\relax\footnotetext{$^1$ pierre.bousquet@math.univ-toulouse.fr, $^2$ gyula.csato@upc.edu}

\section{Introduction}

\subsection{The problem}
In this paper, we study the existence  of solutions of the following equation:
\begin{equation}\label{eq78} \tag{$E_a$}
\begin{cases}
\diw u + \langle a , u \rangle = f & \textrm{ in } \Omega,\\
u=0   & \textrm{ on } \partial \Omega.
\end{cases}
\end{equation}
Here, \(\Omega\) is a bounded open set in \(\R^n\), \(n\geq 1\), \(a:\Omega \to \R^n\) is a vector field and \(f : \Omega \to \R\) is a function. The notation \(\langle a , u \rangle\) refers to the standard scalar product in \(\R^n\). We look for a solution  \(u:\Omega \to \R^n\) in Sobolev spaces or in H\"older spaces.

When \(a=0\), the above equation reduces to the classical \emph{divergence equation}, which has attracted considerable attention. Let us just mention for the moment that the expected regularity of a solution naturally depends on the regularity of the data. For instance, assuming that \(\Omega\) is Lipschitz and \(f\in L^{p}(\Omega)\),  there exists a solution \(u\) in \(W^{1,p}_{0}(\Omega;\R^n)\) if and only if 
\begin{equation}
 \label{eq:int condition a is 0}
  \int_{\Omega} f=0.
\end{equation}
The condition \eqref{eq:int condition a is 0} is closely related to the homogeneous  Dirichlet boundary condition. When \(f\in C^{0,\alpha}(\Omegabar)\) and \(\Omega\) is \(C^{2,\alpha}\) for some \(\alpha\in (0,1)\), then a solution $u$ exists in \(C^{1,\alpha}(\Omegabar)\) with $u=0$ on \(\partial \Omega\),  under the same necessary and sufficient condition \eqref{eq:int condition a is 0}.

When \(a\not=0\), the study of the perturbed equation \eqref{eq78} has been initiated in \cite{Csato-Dacorogna}. Quite surprinsingly, it  was observed  that the lower term \(\langle a , u \rangle\) dramatically modifies the existence theory. Indeed, the condition \eqref{eq:int condition a is 0} does not generalize to some integral condition involving $a,$ \emph{unless $a$ is a gradient}.

When \(a\) is a gradient: \(a=\nabla A\) for some function \(A:\Omega\to \R\),  then  for every \(u\),
\[
\diw u + \langle a , u \rangle = e^{-A}\diw (u e^{A}).
\] 
It follows that for every \(f\), the equation \(\diw u + \langle a , u \rangle = f\) is equivalent to
\[
\diw (u e^{A}) = e^A f,
\]
and the classical theory when \(a\equiv 0\) then applies. In particular, the existence of a solution requires that 
\[
\int_{\Omega}e^{A}f=0.
\]
The aim of the present paper is to obtain  existence and regularity of solutions to the divergence equation with a lower order term \eqref{eq78},   under natural regularity assumptions on the data, in both Sobolev and H\"older spaces, when \(a\) is not a gradient.

The existence problem can be formulated in (at least) three different ways that we now detail.
Let us debote by \(T_a\) the operator \(\diw + \langle a , \cdot \rangle\). Assuming that the right hand side \(f\) belongs to a given Banach space \(Y\), we look for a function \(u\) in a given Banach space \(X\) such that \(T_a(u)=f\). This leads to the first formulation of the problem: \emph{Is  $T_a :X\to Y$  onto ?} 
If the answer is positive, then the open map theorem implies that for every \(f\in Y\), there exists some \(u\) in \(X\) such that \(\|u\|_X\leq C \|f\|_Y\), where \(C>0\) is a constant which depends only on \(X, Y\) and \(a\). 

Let us now assume that such a solution \(u\) exists, which is not unique: for a discussion on the kernel \(N(T_a)\) of \(T_a\), see \cite{Csato-Dacorogna}. Therefore,
that one can choose  \(u\) linearly with respect to \(f\) is not obvious. This is the second way to address the existence problem: \emph{Does there exist a right inverse to} $T_a$ ?
If \(T_a\) is surjective, then \(T_a\) admits a right inverse if and only if the kernel \(N(T_a)\) of \(T_a\) admits a complement in \(X\), see \cite[Theorem 2.12]{Brezis functional a}. 

When \(X=W^{1,p}_0(\Omega;\R^n)\) and \(Y=L^{p}(\Omega)\) for some \(p\in (1,\infty)\), we will obtain a bounded linear operator \(S_a : L^{p}(\Omega)\to W^{1,p}_0(\Omega;\R^n)\) such that \(T_a\circ S_a (f)=f\) for every \(f\in L^{p}(\Omega)\). \emph{A priori}, such an  \(S_a\) depends on the exponent \(p\). We will say that a right inverse \(S_a\) to \(T_a\) is \emph{universal} in the scale of Lebesgue spaces if 
\begin{enumerate}
\item the operator \(S_a\) is well defined on \(\bigcup_{1<p<\infty}L^{p}(\Omega)\) with values into the set \(\bigcup_{1<p<\infty} W^{1,p}_0(\Omega;\R^n)\),
\item  for every \(p\in (1,\infty)\), \(S_a : L^{p}(\Omega) \to W^{1,p}_0(\Omega ; \R^n)\) is continuous. 
\end{enumerate}
We are thus led to the third formulation of the existence problem: \emph{Does there exist a  right inverse to  $T_a$ which is universal in the scale of Lebesgue spaces ?} 
Naturally, one can formulate a similar question in the scale of higher order Sobolev spaces \(W^{k,p}\) and H\"older spaces \(C^{k,\alpha}\), with \(k\in \N\), \(p\in (1,\infty)\) and \(\alpha\in (0,1)\).

\subsection{The main results}

Our first main result answers the three above questions in the scale of Lebesgue spaces:
 
\begin{theorem}\label{th-main}
Let \(\Omega\) be a bounded open Lipschitz set. Let  \(q>n\) and  \(a\in L^{q}(\Omega;\R^n)\) such that \(a\) is not a gradient: there exists no \(A\in W^{1,q}(\Omega)\) such that \(a=\nabla A\).
Then there exists a linear operator
\[ 
S_a : \bigcup_{1<p\leq q}L^{p}(\Omega)\to \bigcup_{1<p\leq q} W^{1,p}_0(\Omega;\R^n)
\] 
such that for every \(1<p\leq q\), the map \(S_{a}\) is continuous from \(L^{p}(\Omega)\) into \(W^{1,p}_0(\Omega;\R^n)\) and
\[
\forall f\in L^{p}(\Omega), \qquad \diw S_a(f) + \langle a, S_a(f) \rangle =f.
\]
\end{theorem}

The assumption on the exponent \(q>n\) is related to  the fact that the lower order term \(\langle a , u \rangle\) is expected to be in \(L^{p}\) for any \(u\in W^{1,p}\). 
When \(p<n\), the Sobolev embedding \(W^{1,p}\subset L^{p^*}\), \(p^*=\frac{np}{n-p}\) suggests that we  should only require \(a\in L^{n}\).  Indeed, under such an assumption, \(\langle a , u \rangle\) belongs to \(L^{p}\) for every \(u\in W^{1,p}_0\). Here, we use the fact that \(\frac{1}{p^*}+\frac{1}{n}=\frac{1}{p}\).
In the above statement, we require the slightly stronger assumption \(a\in L^{q}\) for some \(q>n\).

When \(p>n\), using  the Morrey embedding \(W^{1,p}\subset L^{\infty}\), one can see that the assumption \(a\in L^{p}\) is the natural assumption to ensure that \(\langle a , u \rangle\) belongs to \(L^p\).  Finally, when \(p=n\), the fact that \(W^{1,p}_0\subset \bigcap _{1\leq r<\infty}L^{r}\) shows that the assumption \(a\in L^{q}\) for some \(q>n\) is natural to ensure that \(\langle a , u \rangle\) belongs to \(L^{p}\). 

The operator \(S_a\) is a universal construction in the scale of Lebesgue spaces. In particular, for every \(1<p_1<p_2\), \(S_a|_{L^{p_2}}=(S_{a}|_{L^{p_1}})|_{L^{p_2}}\). As a matter of fact, \(S_a\) does not depend on \(q\), in the following sense:

\begin{remark}\label{remark-185}
If there exists \(r>q\) such that \(a\in L^{r}(\Omega)\), then \(S_a\) maps continuously \(L^{r}(\Omega)\) into \(W^{1,r}_0(\Omega;\R^n)\).
\end{remark}

The above remark can be seen as a regularization property of the  construction given in the proof of Theorem \ref{th-main}. In order to obtain a universal construction in the whole scales of Sobolev and H\"older spaces, we  assume that \(\Omega\) is at least of class \(C^2\).
We use the following notation to abbreviate higher order Sobolev and H\"older spaces with zero boundary values:
$$
  W_z^{m,p}=W^{m,p}\cap W_0^{1,p}\quad\text{ and }\quad C_z^{m,\alpha}=\left\{u\in C^{m,\alpha}:\, u=0\text{ on }\delomega\right\}.
$$
In particular, when \(m\geq 2\), the space \(W_{z}^{m,p}\) does not agree with  $W_0^{m,p}$ which usually denotes the closure of smooth functions with compact support in $\Omega$ in the norm $W^{m,p}$.

\begin{theorem}\label{th-universal-property}
Let \(\Omega\) be a bounded open set of class \(C^2\). Let  \(q>n\) and  \(a\in L^{q}(\Omega;\R^n)\) such that \(a\) is not a gradient. Then there exists an operator 
\[
\overline{S_a}:\bigcup_{1<p\leq q}L^{p}(\Omega)\to \bigcup_{1<p\leq q} W^{1,p}_0(\Omega;\R^n)
\]
such that \(\diw S_a(f) + \langle a, S_a(f) \rangle =f\) for every \( f\in \bigcup_{1<p\leq q} L^{p}(\Omega)\).
Moreover, we have the following additional properties:
\begin{enumerate}
\item We assume that \(\Omega\) is  of class \(C^{m+2}\) and \(a\in W^{m,r}(\Omega;\R^n)\) for some \(m\in \N\) and \(r>\frac{n}{m+1}\). Then for every \(1<p\leq r\),  \(\overline{S_a}\) maps continuously \(W^{m,p}(\Omega)\) into \(W^{m+1,p}_z(\Omega;\R^n)\).
\item We assume that \(\Omega\) is of class \(C^{m+2, \alpha}\) and \(a\in C^{m,\alpha}(\Omegabar;\R^n)\) for some \(m\in \N\) and \(\alpha \in (0,1)\).
Then \(\overline{S_a}\) maps continuously \(C^{m,\alpha}(\Omegabar)\) into \(C^{m+1,\alpha}_z(\Omegabar;\R^n)\). 
\end{enumerate}
\end{theorem}

\begin{remark}\label{remark-regularity-omega}
The regularity assumptions that we make on the data are essentially sharp, except possibly for the set \(\Omega\).
Indeed, in the scale of Sobolev spaces, one expects that the statement holds true for every \(\Omega\) of class \(C^{m+1}\) instead of class \(C^{m+2}\). Similarly, in the scale of H\"older spaces, the conclusion should be correct when \(\Omega\) is merely of class \(C^{m+1,\alpha}\). However,  the proof of Theorem \ref{th-universal-property} relies on the inversion of the divergence (when \(a=0\)) and we are not aware of any \emph{universal} construction of such an inverse under these sharper regularity assumptions on the domain \(\Omega\).
\end{remark}

\subsection{Comparison with previous results}

In \cite{Csato-Dacorogna}, the existence of a solution to \eqref{eq78} is proved when the data are smooth: one assumes that \(\Omega\) is \(C^{r+4}\), \(f\) and \(a\) are \(C^{r+3}\) for some \(r\geq 0\). Moreover, one requires that the domain \(\Omega\) is diffeomorphic to a ball. Finally, the vector field \(a\) must satisfy the following condition:  
\begin{equation}
 \label{theorem:curl a assmumption}
 \operatorname{curl}a(x_0)\neq 0\quad\text{ for some }x_0\in \Omega.
\end{equation}
Under these assumptions, there exists a solution \(u\in C^{r+1}(\overline{\Omega};\R^n)\), see \cite[Theorem 2]{Csato-Dacorogna}.

Under a stronger assumption on the vector field \(a\), namely
\begin{equation}\label{eq92}
\inf_{x\in \partial \Omega}|\textrm{ curl }a(x)| >0,
\end{equation}
a right  inverse to \(T_a\) is constructed in the setting of H\"older spaces, see \cite[Theorem 3]{Csato-Dacorogna}. In the latter statement, the regularity assumptions are sharp for \(f\), but not for \(\Omega\) or \(a\). It also follows from the proof that the construction is universal in the scale of H\"older spaces. 

Both conditions \eqref{theorem:curl a assmumption} and \eqref{eq92}  imply (but are not equivalent to) the fact that \(a\) is not a gradient. In \cite[Theorem 5]{Csato-Dacorogna}, it was observed   that solutions to \eqref{eq78} exist in certain cases even if $\operatorname{curl}a$ vanishes everywhere without any integral condition of $f$, as long as $a$ is not a gradient. In view of Theorem \ref{th-main},
 the latter turns out to be the natural assumption  for the  existence theory of \eqref{eq78}.

\begin{remark} The two results in \cite{Csato-Dacorogna} are stated for a more general boundary condition, given by a vector field \(u_0\). However, this case easily reduces to the case \(u_0\equiv 0\), up to a modification of the right hand side \(f\) (see the first step of the proof of \cite[Theorem 2]{Csato-Dacorogna}).  
\end{remark}

\subsection{Some ideas of the proof}

We follow a totally different approach from the one used in \cite{Csato-Dacorogna}. 
The proof of Theorem \ref{th-main} relies on the fact that \eqref{eq78} is a compact perturbation of the classical divergence equation \(\diw u=f\). In order to be more precise, we need to give a quick review of the case $a=0$. Thus we shall consider solutions $u$ of the problem
\begin{equation}\label{eq:a is zero}\tag{$E_0$}
\begin{cases}
\diw u = f & \textrm{ in } \Omega,\\
u=0   & \textrm{ on } \partial \Omega.
\end{cases}
\end{equation}
For every \(1<p<\infty\), we define $L^p_{\sharp}(\Omega)$ as
$$
  L^p_{\sharp}(\Omega)=\left\{f\in L^p(\Omega):\,\intomega f=0\right\}.
$$
Fix \(p\in (1,\infty)\). As explained in \cite[Lemma 10]{Auscher-Russ-Tchamitchian}, the operator \(T_0=\diw\) defined on \(W^{1,p}_0\) is surjective onto \(L^{p}_\sharp\), provided that the range \(R(T_{0}^*)\) of the dual \(T_{0}^*\) is closed.  Indeed, this condition implies that the range of \(T_{0}\) is equal to \(N(T_{0}^*)^\perp\), see  
 \cite[Theorem 2.19]{Brezis  functional a}. This leads to the desired conclusion since \(N(T_{0}^*)^\perp=L^{p}_{\sharp}\).
 When \(p=2\), the fact  that \(R(T_{0}^*)\) is closed can be obtained  as a consequence of the following estimate \cite[Chapitre 3, Lemme 7.1]{Necas}:
\begin{equation}
\label{eq:Necas estimate}
  \forall f\in L^{2}(\Omega), \qquad \|f\|_{L^{2}} \leq C \left( \|f\|_{W^{-1,2}} + \|\nabla f\|_{W^{-1,2}}\right).
\end{equation}
This strategy to prove the surjectivity of \(T_{0}\)  thus relies on a \emph{duality} estimate. It does not provide a right inverse to \(T_0\). 

The duality approach is also one of the main features of the proof given  by Bourgain and Brezis in \cite[Theorem 2']{Bourgain-Brezis} to establish the existence of a right inverse to \(T_0\) when \(\Omega\) is a Lipschitz set. Indeed, their argument  relies on the following abstract result, see \cite[Lemma 8]{Bourgain-Brezis}:
\begin{lemma}\label{lemma_fa}
Let \(E\) and \(F\) be two Banach spaces  and   let \(T\) be a bounded linear operator from \(E\) into \(F\) such that \(N(T^*)=\{0\}\). Assume that there exists a bounded linear operator \(\Tilde{S}\) from \(F\) to \(E\) and a compact linear operator \(K\) from \(F\) into itself such that \(T\circ \Tilde{S} = I+K\). Then \(T\) admits a right inverse \(S\).
\end{lemma}
The above lemma is then applied to $E=W_0^{1,p}(\Omega;\re^n),$ $F=L^{p}_{\sharp}(\Omega)$ and $T=T_0=\diw,$ where the condition $N(T_{0}^{\ast})=\{0\}$ follows from the fact that $T_{0}^{\ast}$ acts on $L^p_{\sharp}$ (and not $L^p$).  

The strategy adopted in \cite{Bourgain-Brezis} can be adapted in various settings. It can be exploited in any higher order Sobolev spaces or H\"older spaces, see \cite{HPNG-thesis}, to get an existence theory for \eqref{eq:a is zero} under sharp regularity assumptions on the domain \(\Omega\). For the equation \eqref{eq78}, we will heavily rely on a minor adaptation of the proof of Lemma \ref{lemma_fa} (with $T=T_a=\diw\cdot+\langle a,\cdot\rangle$).

To the best of our knowledge, the right inverse \(S\) constructed in the proof of \cite[Theorem 2']{Bourgain-Brezis} depends on the exponent \(p\). In order to get a \emph{universal} right inverse to the divergence operator, at least in the scale of Lebesgue spaces, one can rely on the construction due to Bogovski \cite{Bogovski}, see also \cite[Theorem 4.1]{Acosta-Duran-Muschietti}:

\begin{theorem}\label{th-Bogovski}
Assume that \(\Omega\) is a Lipschitz set. Then there exists a linear operator
\[
S_0 : \bigcup_{1<p<\infty} L^{p}_\sharp(\Omega) \to \bigcup_{1<p<\infty}W^{1,p}_0(\Omega;\R^n)
\]
such that \(T\circ S_0 (f)= f\) for every \(f\in \bigcup_{1<p<\infty} L^{p}_\sharp(\Omega)\). Moreover, for every \(p\in (1,\infty)\), \(S_0\) is continuous from \(L^{p}_\sharp(\Omega)\) into \(W^{1,p}_0(\Omega;\R^n)\).
\end{theorem}

The proof of Theorem \ref{th-main} relies on the existence of such an \(S_0\). We first observe that \(T_a\circ S_0 = I +K\) where \(I\) is the identity on \(L^{p}_{\sharp}\) and \(K\) is a compact operator from \(L^{p}_{\sharp}\) into itself. We then prove that \(N(T_{a}^*)=\{0\}\). We next apply Lemma \ref{lemma_fa} with \(T=T_a\) and \(\Tilde{S}=S_0\) to get the desired operator \(S_a\). In order to check that \(S_a\) has the universal properties stated in Theorem \ref{th-main}, we exploit the universal property of the operator \(S_0\) given by Theorem \ref{th-Bogovski}. We also need to detail (and slightly adapt) the explicit construction in the proof of Lemma \ref{lemma_fa}, to ensure that the resulting operator \(S_a\) still possesses the universal property in the scale of Lebesgue spaces.

The operator \(S_0\) given by Theorem \ref{th-Bogovski} maps \(W^{k,p}\) into \(W^{k+1,p}_0\) for every \(k\in \N\). However,  \(W^{k+1,p}_0\), which is the closure of \(C^{\infty}_c\) in \(W^{k+1,p}\), is strictly contained in \(W^{k+1,p}_z\) defined as the intersection \(W^{1,p}_0\cap W^{k+1,p}\). Hence, we can not use this map \(S_0\) in the setting of Theorem \ref{th-universal-property}. In the framework of higher order Sobolev spaces and H\"older spaces, we will rely instead on the following construction \cite[Theorem 9.2, Remark 9.3]{Csato-Dacorogna-Kneuss}, which requires that \(\Omega\) be at least \(C^2\).

\begin{theorem}
\label{theorem:a is zero higher regularity}
Assume that $\Omega$ is a bounded open set of class \(C^2\). 
Then there exists a linear continuous operator $\overline{S_0}$ satisfying the same conclusion as in Theorem \ref{th-Bogovski} with the following additional properties:
\begin{enumerate}
\item We assume that \(\Omega\) is  of class \(C^{m+2}\) for some \(m\in \N\). Then for every \(p\in (1,\infty)\),  \(\overline{S_0}\) maps continuously \((W^{m,p}\cap L^{1}_{\sharp})(\Omega)\) into \(W^{m+1,p}_z(\Omega;\R^n)\).
\item We assume that \(\Omega\) is of class \(C^{m+2, \alpha}\) for some \(m\in \N\) and \(\alpha \in (0,1)\).
Then \(\overline{S_0}\) maps continuously \((C^{m,\alpha}\cap L^{1}_{\sharp})(\overline{\Omega})\) into \(C^{m+1,\alpha}_z(\overline{\Omega};\R^n)\). 
\end{enumerate}
\end{theorem}

In the proof of Theorem \ref{th-universal-property}, the above map \(\overline{S_0}\) plays a crucial role, in a similar way as the operator \(S_0\) is used 
in the proof of Theorem \ref{th-main}. We also rely on standard properties of the pointwise multiplication  in  higher order Sobolev spaces and in  H\"older spaces.

\subsection{Plan of the paper}
 
The next section is dedicated to the proof of Theorem \ref{th-main} while Theorem \ref{th-universal-property} is proven in Section 3. In the last section, we discuss the non-existence in $L^1$ and $L^{\infty}.$ In contrast to the previous existence and regularity results, the proof is exactly the same as for the case $a=0.$ Finally, for the convenience of the reader, we have gathered in the Appendix some technical tools.

\section{Construction of \(S_a\) on \(L^{p}\)}

The construction of $S_a$ is inspired from the proof of Lemma \ref{lemma_fa} in \cite{Bourgain-Brezis}.
In our setting, we take \(E=W^{1,p}_0(\Omega;\R^n)\), \(F=L^{p}(\Omega)\) for some \(p\in (1,\infty)\) and  
\begin{equation}
 \label{eq:T definition}
   T_a(u)=\diw u + \langle a , u \rangle.
\end{equation}
Throughout this section, we assume that \(\Omega\) is Lipschitz.
We first observe that $T_a$ is continuous:

\begin{lemma}
\label{lemma:T continuous}
Let $T_a$ be given by \eqref{eq:T definition}. If $q> n,$  $a\in L^q(\Omega),$  then for every $1\leq p\leq q$,  $T_a$ maps continously \(W^{1,p}(\Omega)\) into $L^p(\Omega)$.
\end{lemma}

\begin{proof}
We recall that for two functions $f\in L^r$ and $g\in L^s$,  we have 
\begin{equation}
 \label{eq:multipliciation in Lp}
  fg\in L^t\quad\text{ for all $t$ satisfying }\quad \frac{1}{r}+\frac{1}{s}\leq \frac{1}{t}.
\end{equation}
We only need to show that $\langle a,u\rangle\in L^p$ with the appropriate estimate. If $p<n$ then by the Sobolev embedding,  $u\in L^{p^*}$ with $p^*=\frac{np}{n-p}$. Using  \eqref{eq:multipliciation in Lp} and the assumption $q\geq n$, it follows that $\langle a,u\rangle\in L^p$. For  the case $p=n$, one  uses that $W^{1,n}$ embeds into any $L^s$ as long as $s<\infty$ and one applies again \eqref{eq:multipliciation in Lp} together with the fact that $q>n=p$. In the third case $n<p\leq q$, one relies on the  Morrey embedding \(W^{1,p}\subset L^{\infty}\) to conclude.
\end{proof}

\smallskip

The next lemma  is the key tool to prove that $N(T_{a}^{\ast})=\{0 \}.$ This is the step where the assumption that \emph{\(a\) is not a gradient} plays a crucial role.

\begin{lemma}\label{lemma_kernel}
Let $\Omega\subset\re^n$ be a bounded  open Lipschitz set and  \(a\in L^{q}(\Omega;\R^n)\) for some \(q\in [1,\infty]\). We assume that \(a\) is not a gradient: there exists no $A\in W^{1,q}(\Omega)$ such that $\nabla A=a.$ If \(g\in L^{q'}(\Omega)\) is such that \(\nabla g =g a\) in the sense of distributions, i.e.
$$
  \intomega g\diw u =-\intomega \langle g a,u\rangle\quad \text{for all }u\in C_c^{\infty}(\Omega,\re^n),
$$
then 
$$
  g\equiv 0.
$$
\end{lemma}

\begin{remark}
The proof of this lemma is much simpler if one assumes $a\in C^0(\Omega;\re^n).$ Applying repeatedly the Sobolev and Morrey embeddings,  one obtains that $g\in C^1$ and $\nabla g=ga$ in a classical sense. 
Let $x,y\in\Omega$ and assume that $\gamma$ is a $C^1$ curve connecting $x$ and $y$ with $\gamma(0)=x$ and $\gamma(1)=y.$ Thus $g\circ\gamma$  satisfies the differential equation
$
(g\circ\gamma)'=\langle \nabla g\circ\gamma,\gamma'\rangle=(g\circ\gamma)\langle a\circ\gamma,\gamma'\rangle
$
and it follows that
\begin{equation*}
  g(y)=g(x)e^{\int_0^1\langle a(\gamma(t)),\gamma'(t)\rangle dt}.
\end{equation*}
This implies that either $g\equiv 0$ or $g(x)\neq 0$ for all $x\in\Omega$. But the second case cannot occur, because if $g$ never vanishes, then $A=\ln |g| \in C^1(\Omega)$ satisfies
$\nabla A=\nabla g/g=(ga)/g=a,$
which is a contradiction to the hypothesis on $a.$
\end{remark}

\begin{proof} Without loss of generality, one can assume that \(\Omega\) is connected.
The equality \(\nabla g =g a\) implies that \(\nabla g \in L^{1}(\Omega)\) and thus \(g\in W^{1,1}(\Omega)\).
Let \(Q\) be a cube contained in \(\Omega\). Up to a dilation and an isometry, we can assume that \(Q=(0,1)^n \subset \Omega\).
For almost every \(x'\in (0,1)^{n-1}\), the map \(g_{x'}:x_n\mapsto g(x',x_n)\) belongs to \(W^{1,1}(0,1)\) while \(x_n\mapsto a_n(x',x_n)\) belongs to \(L^{1}(0,1)\) and moreover 
\[
g_{x'}'(x_n) = a_n(x',x_n)g_{x'}(x_n).
\]
It thus follows that for such \(x'\), for every \(x_n, y_n \in [0,1]\),
\begin{equation}\label{eq590}
g(x',y_n)=g(x',x_n) e^{\int_{x_n}^{y_n}a_n(x',t)\,dt}.
\end{equation}
Repeating the above argument in every direction parallel to the coordinate axes between
\[
(x_{1}, \dots, x_{i-1}, x_{i}, y_{i+1}, \dots, y_n) \textrm{ and } (x_{1}, \dots, x_{i-1}, y_{i}, y_{i+1}, \dots, y_n), \qquad 1\leq i \leq n,
\]
we deduce that for a.e. \(x,y\in Q\), \(g(y)=g(x)e^{D(y,x)}\) with
\begin{equation}\label{eq589}
D(y,x)=\sum_{i=1}^n \int_{x_i}^{y_i}a_i(x_1,\dots, x_{i-1}, t, y_{i+1}, \dots , y_n)\,dt.
\end{equation}
In particular, either \(g\equiv 0\) on \(Q\) or \(g>0\) a.e. on \(Q\) or \(g<0\) a.e. on \(Q\).

Since for every two cubes \(Q_1, Q_2 \subset \Omega\) such that \(Q_1\cap Q_2 \not=\emptyset\), the \emph{same} conclusion among the three above alternatives must hold true, the connectedness of \(\Omega\) implies that either \(g>0\) a.e. on \(\Omega\), or \(g<0\) a.e. on \(\Omega\) or \(g\equiv 0\) on \(\Omega\).

Assume by contradiction that \(g>0\) a.e. on \(\Omega\) and consider again the cube \(Q=(0,1)^n \subset \Omega\).
By the Fubini theorem and the fact that the function \(D\) defined in \eqref{eq589} belongs to \(L^{1}(Q\times Q)\), for a.e. \(x\in Q\), the function \(y\mapsto D(y,x)\)  belongs to \(L^{1}(Q)\). We fix such an \(x\) for which we further require  that \(g(x)>0\). Then the identity \(g(y)=g(x)e^{D(y,x)}\) shows that \(\ln g \in L^{1}(Q)\).

We claim that for every \(\varphi \in C^{\infty}_c(Q)\) and \(1\leq i \leq n\),
\begin{equation}\label{eq607}
\int_{Q} \ln g \, \partial_i \varphi = -\int_{Q}  a_i \varphi.
\end{equation}
Let us prove the claim for \(i=n\). By the Fubini theorem,
\begin{align*}
\int_{Q} \ln g \, \partial_n \varphi &= \int_{(0,1)^{n-1}}\,dx' \int_{0}^1 \ln g (x',t_n) \, \partial_n \varphi(x',t_n)\,dt_n\\
& =  \int_{(0,1)^{n-1}}\,dx'\int_{0}^1 \ln g (x',0)\,  \partial_n \varphi(x',t_n)\,dt_n \\
&\qquad  + \int_{(0,1)^{n-1}}\,dx'\int_{0}^1 \left(\int_{0}^{t_n}a_n(x',s)\,ds\right) \partial_n \varphi(x',t_n)\,dt_n
\end{align*}
where the last line follows from \eqref{eq590} with \(x_n=0\) and \(y_n=t_n\).  Now, by the Fubini theorem,
\[
\int_{0}^1 \left(\int_{0}^{t_n}a_n(x',s)\,ds\right) \partial_n \varphi(x',t_n)\,dt_n = -\int_{0}^{1} \varphi(x',s)a_n(x',s)\,ds. 
\]
Since 
\[
\int_{0}^1 \ln g (x',0)\, \partial_n \varphi(x',t_n)\,dt_n = \ln g(x',0) \int_{0}^1  \partial_n \varphi(x',t_n)\,dt_n=0,
\]
we get
\[
\int_{Q} \ln g \, \partial_n \varphi = -\int_{Q}a_n \varphi.
\]
We can repeat this calculation in every direction \(i=1, \dots n\) by using the identity corresponding to \eqref{eq590} where \(n\) is replaced by \(i\). This proves claim \eqref{eq607}. We deduce therefrom  that \(\ln g \in W^{1,q}(Q)\) with \(\nabla \ln g = a\).

Since this is true for every cube \(Q \subset \Omega\), this implies that \(a=\nabla (\ln g)\) on \(\Omega\), which contradicts the fact that \(a\) is not a gradient on \(\Omega\). Hence, we cannot have \(g>0\) a.e. The case \(g<0\) a.e. can be treated similarly. This proves that \(g\equiv 0\) as desired. 
\end{proof}

We proceed to explain how the above lemma implies that \(N(T_{a}^*)=\{0\}\).
\begin{lemma}\label{lm-nta-t}
Let $q>n$ and suppose that $a\in L^q(\Omega;\re^n)$ is  not the gradient of a $W^{1,q}$ function. Then for every \(1<p\leq q\), the operator \(T_a : u\in W^{1,p}(\Omega;\R^n) \mapsto \diw u + \langle a , u \rangle \in L^{p}(\Omega)\) satisfies: 
\[
N(T_{a}^*)=\{0\}.
\]
\end{lemma}
\begin{proof}
Let \(g\in (L^{p})^*\) such that \(T_{a}^*(g)=0\). Identifying \(g\) with an element of \(L^{p'}\), \(p'=\frac{p}{p-1}\), this means that 
\[
\forall u \in W^{1,p}_0, \qquad \int_{\Omega}g(\diw u + \langle a , u \rangle)=0.
\]
Note that $a\in L^q\subset L^p$. Moreover,  \(a\) is not the gradient of a $W^{1,p}$ function for otherwise, there would exist \(A\in W^{1,p}\) such that \(a=\nabla A\). This would imply that \(\nabla A \in L^{q}\) and thus by the Poincar\'e-Wirtinger inequality, \(A\in W^{1,q}\), a contradiction to the assumption on \(a\).  
In view of Lemma \ref{lemma_kernel} applied with \(p\) instead of \(q\), we deduce that \(g=0\). This proves that \(N(T_{a}^*)=0\).
\end{proof}

By the Hahn-Banach theorem, see e.g. \cite[Corollary 1.8]{Brezis functional a}, this implies that the range \(R(T_a)\) of \(T_a\) is dense in \(L^{p}(\Omega)\). We shall see later that in fact \(R(T_a)=L^{p}(\Omega)\).

We now introduce two operators that will play a crucial role in the sequel.
\begin{definition}
\label{definition of S and K}
Let $\Omega\subset\re^n$ be a bounded open Lipschitz set. Let $S_0$ be the map given by Theorem \ref{th-Bogovski}.
We then define for \(f\in \bigcup_{1<p<\infty}L^{p}(\Omega)\): 
\begin{equation}\label{eq_def_S}
S(f) = S_0\left(f-\fint_{\Omega}f\right), 
\end{equation}
\begin{equation}\label{eq_def_K}
K(f)=- \fint_{\Omega}f +\langle a ;  S(f) \rangle.
\end{equation}
\end{definition}
For every \(1<p<\infty\), \(S\) defines a continuous linear map from \(L^{p}(\Omega)\) into \(W^{1,p}_0(\Omega;\re^n)\).

The map $K$ is continuous as well, but also  compact under an appropriate assumption on $a.$

\begin{lemma}\label{lemma_compacity_K}
Assume that \(a\in L^{q}(\Omega;\re^n)\) for some \(q>n\).
For every \(1<p\leq q\), the map \(K\) is a compact linear map from \(L^{p}(\Omega)\) into \(L^{p}(\Omega)\).
\end{lemma}
\begin{proof}
 We define the exponent \(r\in (1,\infty]\) by \(\frac{1}{r}+\frac{1}{q}=\frac{1}{p}\). When \(p<n\), the fact that \(q>n\) implies that \(r< p^*=\frac{np}{n-p}\). It follows that the embedding \(W^{1,p}_0\subset L^{r}\) is compact and this remains true when \(p\geq n\).  
By the H\"older inequality,  the map 
\[
u\in L^r(\Omega;\R^n) \mapsto \langle a ; u \rangle \in L^p(\Omega) 
\]
is continuous. It thus follows  that the map
\[
u\in W^{1,p}_0(\Omega;\R^n) \mapsto \langle a ; u \rangle \in L^p(\Omega)
\]
is compact as the composition of a continuous operator with a compact one. Composing again by the continuous operator \(S\), we infer that the map \(f\in L^p \mapsto \langle a , S(f)\rangle \in L^p\) is compact.  The map $f\mapsto \fint_{\Omega} f$ is clearly compact from $L^p$ to $L^p.$ It follows that \(K\) is compact as well.
\end{proof}
\smallskip

We observe that  for every \(f\in \bigcup_{1<p\leq q}L^{p}(\Omega)\),
\begin{equation}
\label{eq:TS is I plus K}
\begin{split}
T_a\circ S (f) &= \diw S_0\left( f-\fint_{\Omega}f\right) +\langle a ;  S_0\left( f- \fint_{\Omega}f\right)\rangle\\
&=f-\fint_{\Omega}f +\langle a ;  S_0\left( f- \fint_{\Omega}f\right)\rangle=(I+K)(f).
\end{split}
\end{equation}

\begin{remark}\label{remark-498}
In view of Lemma \ref{lm-nta-t} and Lemma \ref{lemma_compacity_K}, one can apply  Lemma \ref{lemma_fa}  to \(\widetilde{S}=S\), where \(S\) is defined in  \eqref{eq_def_S}. We thus obtain that \(T_a:W^{1,p}_0(\Omega;\R^n) \to L^{p}(\Omega)\) has a right inverse, provided that \(a\) is not a gradient. However, we shall not use Lemma \ref{lemma_fa} directly, but slightly modify its original proof from \cite{Bourgain-Brezis} to obtain a \emph{universal} construction, first in the whole scale of Lebesgue spaces (this will imply Theorem \ref{th-main}), and next, in higher order Sobolev and H\"older spaces (to get Theorem \ref{th-universal-property}).
\end{remark}

The kernel of the operator \(I+K\) only contains \(L^q\) functions provided that \(a\in L^q\). More generally,

\begin{lemma}\label{lm_inva_p}
Let \(q>n\) and \(a\in L^{q}(\Omega;\re^n)\). Then for every \(1<p\leq q\) and every \(f\in L^{p}(\Omega)\), if \((I+K)(f)\in L^{q}(\Omega)\), then \(f\in L^{q}(\Omega)\).
\end{lemma}

\begin{proof}
Let \(f\in L^{p}\) such that
\[
\widetilde{f} +\langle a , S_0(\widetilde{f})\rangle \in L^{q}(\Omega), \qquad \widetilde{f}=f-\fint_{\Omega}f.
\]
We now distinguish three cases:
\smallskip

\textit{Case 1.} 
Assume first that \(p>n\). Since \(S_0\) maps \(L^{p}_\sharp\) into \(W^{1,p}\), it follows from the Morrey embedding that  \(S_0(\widetilde{f})\in L^{\infty}\). Since \(a\in L^{q}(\Omega)\), we  have \(\langle a , S_0(\widetilde{f})\rangle \in L^{q}(\Omega)\), which  completes the proof in that case.
\smallskip

\textit{Case 2.}
When \(p=n\), we use that \(S_0(\widetilde{f})\in W^{1,n}\). The latter space being contained in \( \bigcap_{1\leq r <\infty} L^{r}\), we deduce that
\[
\langle a , S_0(\widetilde{f})\rangle \in \bigcap_{1\leq r <q}L^r.
\]
Since $q>n,$ we can choose some $r\in (n,q)$ such that $\langle a , S_0(\widetilde{f})\rangle \in L^r$. Hence $\widetilde{f}\in L^r$ and we are thus reduced to the first case.

\smallskip

\textit{Case 3.}
Finally, if \(p<n\), we rely on the fact that \(S_0\) maps \(L^{r}_\sharp\) into \(W^{1,r}\) for every \(1<r<\infty\) and the Sobolev embedding \(W^{1,r}\subset L^{r^*}\), with \(\frac{1}{r^*}=\frac{1}{r}-\frac{1}{n}\).  For \(r=p\), this first implies that \(S_0(\widetilde{f})\in L^{p^*}\). Since \(a\in L^{q}\) with \(q>n\), if follows that \(\langle a , S_0(\widetilde{f})\rangle \in L^{p_1}\) with
\[
\frac{1}{p_1}=\frac{1}{p^*} + \frac{1}{q} = \frac{1}{p} + \frac{1}{q}-\frac{1}{n}.
\]
Since \(p_1<q\), one also has 
\begin{equation}\label{eq540}
\widetilde{f}\in L^{p_1}.
\end{equation}
Since \(\frac{1}{q}-\frac{1}{n} <0\), one gets that \(p_1>p\). If \(p_1<n\), one can repeat this argument with \(r=p_1\). This proves that \(\widetilde{f}\in L^{p_2}\) with
\[
\frac{1}{p_2}= \frac{1}{p_1} + \frac{1}{q}-\frac{1}{n}=\frac{1}{p} + 2\left(\frac{1}{q}-\frac{1}{n}\right).
\]
Let us define the 
sequence \((p_k)_{k\in \N}\) as 
$$
  \frac{1}{p_k}=\frac{1}{q}-\frac{1}{n}+\frac{1}{p_{k-1}}=\frac{1}{p}+k\left(\frac{1}{q}-\frac{1}{n}\right).
$$
We observe that if \(1<p_{k-1}<n\), then \(\frac{1}{p_k}>\frac{1}{q}\) and thus \(p_k>0\). Since \(\frac{1}{p_k}<\frac{1}{p_{k-1}}\), it follows that \(p_k>1\). Moreover, the sequence \((\frac{1}{p_k})_{k\in \N}\) tends to \(-\infty\). Hence, there exists \(k_0\in \N\) such that \(1<p_{k}<n\) for every \(k\leq k_0\) and \(p_{k_{0}+1}\geq n\). 
Bootstrapping the  argument leading to \eqref{eq540}, we deduce that \(\widetilde{f}\)  belongs to \(L^{p_{k_0+1}}\). The two first parts of the proof then apply to yield \(\widetilde{f}\in L^{q}\). 
\end{proof}

\smallskip

According to Remark \ref{remark-498}, we need to check that the construction described in the  original proof of Lemma \ref{lemma_fa}, see \cite{Bourgain-Brezis},  can be slightly adapted to yield a \emph{universal} right inverse to \(T_a\). A first tool is provided by the following:

\begin{lemma}
\label{lemma:decomp of Lp in Xp plus N}
Let $q>n$ and suppose that $a\in L^q(\Omega;\re^n).$ 
\smallskip 

(i) Then there exists a closed subspace $X=X_q$ of $L^1(\Omega)$ (which is independent of $p$) such that for every $1<p\leq q$
$$
  L^p(\Omega)=(X\cap L^p(\Omega))\oplus N,
$$
where $N=\{f\in L^q(\Omega):\, (I+K)(f)=0\}.$ 
\smallskip

(ii) If in  addition $a\in L^r(\Omega;\re^n)$ for some $r\geq q,$ then \(N\subset L^{r}(\Omega)\) and for the same $X=X_q$ of (i), one has 
$$
  L^r(\Omega)=(X\cap L^r(\Omega))\oplus N.
$$
\end{lemma}

\begin{proof}
(i) By Lemma \ref{lemma_compacity_K}, the map \(K\) is  a compact operator from \(L^{q}\) into \(L^{q}\). This implies that
\(N\) is a finite dimensional subspace of \(L^{q}\), and thus, a finite dimensional subspace of \(L^{1}\). Hence, there exists a closed subspace \(X\) of \(L^{1}\) such that
\[
L^{1}=X \oplus N.
\] 
For every \(1<p\leq q\), the fact that \(N\subset L^p\) implies that
\[
L^{p}=(X\cap L^p) \oplus N.
\]
Note that $X\cap L^p$ is a closed subspace of $L^p.$
\smallskip

(ii) First observe that if $a\in L^r,$ then by Lemma \ref{lm_inva_p} with \(r\) instead of \(q\), one has \(N\subset L^{r}\). We can thus conclude as in (i) that \(L^{r}=(X\cap L^r) \oplus N\).
\end{proof}

\begin{lemma}
\label{lemma:decomp of Lp in range plus Z}
Let $q>n$ and suppose that $a\in L^q(\Omega;\re^n)$ is  not the gradient of a $W^{1,q}$ function. Then there exists a finite dimensional space $Z=Z_q\subset T_a(C_c^{\infty}(\Omega;\re^n))$ such that for all $1<p\leq q$
\begin{equation}
 \label{eq369}
  (I+K)(L^p(\Omega))\oplus Z=L^p(\Omega).
\end{equation}
\end{lemma}

\begin{proof}
By Lemma \ref{lm-nta-t}, \(\{f\in L^{q'} : T_{a}^*(f)=0\}=\{0\}\), which implies that \(T_{a}(W^{1,q}_0)\) is dense in \(L^{q}\) by the Hahn-Banach theorem. Since \(C^{\infty}_c\) is dense in \(W^{1,q}_0\), it follows that \(T_{a}(C^{\infty}_c)\) is dense in \(L^q\).
Moreover, \((I+K)(L^q)\) has finite codimension in \(L^{q}\), because $K:L^q\to L^q$ is compact, see \cite[Theorem 6.6]{Brezis functional a}. 
In view of Lemma \ref{lemma_complementing_density} in the Appendix,  there exists a finite dimensional space \(Z\subset T_{a}(C^{\infty}_c)\) such that \(Z\oplus (I+K)(L^{q})=L^{q}\). We claim that $Z$ has the desired property.

We  first  prove that $(I+K)(L^p)\cap Z=\{0\}.$ Indeed, let \(f\in L^{p}\) such that \((I+K)(f)\in Z\). Since \(Z\subset L^q\), Lemma \ref{lm_inva_p} implies that \(f\in L^{q}\). Hence, \((I+K)(f)\in Z\cap (I+K)(L^q)=\{0\}\).  This proves that $(I+K)(L^p)\cap Z=\{0\}.$ 

Since $K:L^p\to L^p$ is compact, $(I+K)(L^p)$ is closed in $L^p,$ see \cite[Theorem 6.6]{Brezis functional a}. Using that \(Z\) is finite dimensional, it follows that \((I+K)(L^p)\oplus Z\) is closed in \(L^p\). Since \((I+K)(L^p)\oplus Z\) contains \((I+K)(L^q)\oplus Z = L^q\) which is dense in \(L^{p}\), we deduce that \(Z\oplus (I+K)(L^p)=L^{p}\), as desired.
\end{proof}

\smallskip

Having prepared all the necessary ingredients,  we  can now complete the

\begin{proof}[Proof of Theorem \ref{th-main}]

\textit{Step 1.} 
Fix \(p\in (1,q]\). By Lemma \ref{lm_inva_p}, \(\{f\in L^{p} : (I+K)(f)=0\}=\{f\in L^{q} : (I+K)(f)=0\}\). By construction of \(X\) and Lemma \ref{lemma:decomp of Lp in Xp plus N} (i), the map \(I+K\) defines an isomorphism from \(X\cap L^p\) onto \((I+K)(L^p)\) which are two Banach spaces, as two closed subsets of Banach spaces. We denote by \(V_p : (I+K)(L^p) \to X\cap L^p\) its inverse (continuous by the inverse mapping theorem) and we have the following diagram
$$
\begin{array}{cccccc}
  L^p & = &        X\cap L^p           & \oplus & N &\phantom{aaaaaaa}\text{ where }N=N(I+K)\\ 
      &   & \;\vert \quad\quad\; \uparrow  &        &     &     \\
      &   & I+K \quad V_p         &        &          & \phantom{aaaaaa} V_p=\left((I+K)\big|_{X\cap L^p}\right)^{-1}      \\
      &   & \,\downarrow \quad\quad\; \vert &         &    &      \\
  L^p & = &  (I+K)(L^p)           &   \oplus& Z .      &
\end{array}
$$
We claim that for every \(1<p_1\leq p_2 \leq q\) 
$$
  V_{p_1}(f)=V_{p_2}(f) \quad\text{ for every }f\in (I+K)(L^{p_2}).
$$
In fact, for every \(f\in (I+K)(L^{p_2})\), let \(g\) be the unique element in \(X\cap L^{p_2}\) such that \((I+K)(g)=f\). Then \(V_{p_2}(f)=g\). Moreover,  \(g\in X\cap L^{p_1}\) so that \(V_{p_1}(f)=g\).

In view of Lemma \ref{lemma:decomp of Lp in range plus Z},  for every \(1<p\leq q\), \(L^{p}= (I+K)(L^p)\oplus Z\). Hence,  there exist two continuous projections 
\[
Q_p :  L^{p} \mapsto (I+K)(L^p),
\]
\[
\zeta_p : L^{p} \mapsto Z,
\]
such that for every \(f\in L^{p}\), \(f= Q_p(f) + \zeta_p(f)\).

Let \((e_\alpha)_{\alpha}\) be a finite basis of \(Z\). Since \(Z\subset T(C^{\infty}_c)\), there exists \((\bar{e}_\alpha)\subset C^{\infty}_c\) such that \(T(\bar{e}_\alpha)=e_\alpha\) for every \(\alpha\). Let \((e_{\alpha}^*)_{\alpha}\subset Z^*\) be the dual basis of \((e_\alpha)_\alpha\), which has the property
$$
  \sum_{\alpha}e_{\alpha}^*(g)e_{\alpha}=g\quad\text{ for all }g\in Z.
$$

At last we define for every \(1<p\leq q\), for every \(f\in L^p\),
\[
S_{a}^p(f) = S\circ V_p \circ Q_p (f)+ \sum_{\alpha} e_{\alpha}^* \circ \zeta_p(f)\bar{e}_\alpha.
\]
Remember that \(S\) is given by \eqref{eq_def_S}.

Then, by composition, \(S_{a}^p\) is a linear continuous map from \(L^{p}\) into \(W^{1,p}_0\). Moreover, by \eqref{eq:TS is I plus K},  \(T_a\circ S = I+K\), and thus
\begin{align*}
T_a\circ S_{a}^p(f) &= T_a\circ \left(S\circ V_p \circ Q_p(f) + \sum_{\alpha} e_{\alpha}^* \circ \zeta_p(f)\bar{e}_\alpha\right)\\
&=T_a\circ S\circ V_p \circ Q_p(f) + \sum_{\alpha} e_{\alpha}^*\circ \zeta_p(f)T(\bar{e}_\alpha)\\
&=(I+K) \circ V_p \circ Q_p(f) + \sum_{\alpha} e_{\alpha}^*\circ \zeta_p(f) e_\alpha\\
&=Q_p(f) + \zeta_p(f)=f.
\end{align*}

\textit{Step 2.} The above construction is universal: if \(1<p_1\leq p_2\leq q\), then for every \(f\in L^{p_2}\), we claim that \(S_{a}^{p_1}(f)=S_{a}^{p_2}(f)\). Indeed, one can write \(f= Q_{p_2}(f)+\zeta_{p_2}(f)\) according to the decomposition \(L^{p_2}=(I+K)(L^{p_2}) \oplus Z\). Since \((I+K)(L^{p_2})\subset (I+K)(L^{p_1})\), this is also the decomposition of \(f\) in \(L^{p_1}=(I+K)(L^{p_1}) \oplus Z\):
\[
Q_{p_2}(f)=Q_{p_1}(f) \qquad \textrm{ and } \qquad \zeta_{p_2}(f)=\zeta_{p_1}(f).
\]
We have already checked that \(V_{p_1}\) and \(V_{p_2}\) agree on \((I+K)(L^{p_2})\). Hence,
\begin{align*}
S_{a}^{p_1}(f) &= S\circ V_{p_1} \circ Q_{p_1}(f) + \sum_{\alpha} e_{\alpha}^{*}\circ \zeta_{p_1}(f)\bar{e}_{\alpha}\\
&=S\circ V_{p_2} \circ Q_{p_2}(f) + \sum_{\alpha} e_{\alpha}^{*}\circ \zeta_{p_2}(f)\bar{e}_\alpha\\
&=S_{a}^{p_2}(f).
\end{align*}
It follows that we can define the map \(S_a : \bigcup_{1<p\leq q}L^{p} \to \bigcup_{1<p\leq q}W^{1,p}_0\) by setting \(S_{a}|_{L^{p}}=S_{a}^p\).
\end{proof}
\smallskip

We now prove that the above construction of \(S_{a}\) does not depend on \(q\), see Remark \ref{remark-185}:

\begin{lemma}
\label{lemma:q replace by r}
If \(a\in L^{r}(\Omega)\) for some \(r\geq q\), then \(S_a\) maps continuously \(L^{r}(\Omega)\) into \(W^{1,r}_{0}(\Omega)\).
\end{lemma}

\begin{remark}
\label{remark:unicity of Sa}
Given some $a\in L^r(\Omega;\re^n),$ the map $S_a$ is uniquely determined by the choice of $S_0,$ by the choice of some $q\leq r$ and $X=X_q$ (in Lemma \ref{lemma:decomp of Lp in Xp plus N}), and finally by the choice of the $\overline{e}_{\alpha}$ (which in turn depends on the choice of $Z=Z_q$).
\end{remark}

Lemma \ref{lemma:q replace by r} is an easy consequence of the following more general result which will be needed in the next section: 

\begin{lemma}\label{lm_univ_prop} Let $\Omega\subset\re^n$ be a bounded open Lipschitz set.
Let $q>n$ and $a\in L^q(\Omega;\re^n)$. Let  $1<p\leq q$ and  \(E\), \(F\)  be two Banach spaces continuously embedded in \(W^{1,p}_0(\Omega;\R^n)\) and \(L^{p}(\Omega)\) respectively. We assume that 
\begin{enumerate}
\item The set \(E\) contains \(C^{\infty}_c(\Omega;\R^n)\).
\item For every \(\varphi\in C^{\infty}_c(\Omega;\R^n)\), \(\diw \varphi+\langle a, \varphi \rangle\) belongs to \(F\).
\item For every \((u,f)\in  L^{p}(\Omega)\times F\), 
\begin{equation}\label{eq_univ_prop}
(I+K)(u)=f \quad\Rightarrow\quad u\in F.
\end{equation} 
\item The map \(S\) defined in \eqref{eq_def_S} maps continuously \(F\) into \(E\).
\end{enumerate} 
Then the map \(S_a\) constructed in Theorem \ref{th-main} maps continuously \(F\) into \(E\).
\end{lemma}

We obtain Lemma \ref{lemma:q replace by r} from Lemma \ref{lm_univ_prop} by taking \(p=q\), \(E=W^{1,r}_0\) and \(F= L^{r}\). The assumption \eqref{eq_univ_prop} is satisfied thanks to Lemma \ref{lm_inva_p}.

\begin{proof}[Proof of Lemma \ref{lm_univ_prop}]
By the closed graph theorem, we only need to prove that \(S_a(F)\subset E\).
By construction, for every \(f\in L^{p}\),
\[
S_a(f) = S\circ V_p \circ Q_p(f) + \sum_{\alpha}e_{\alpha}^* \circ \zeta_p(f)\bar{e}_\alpha.
\]
Each \(\bar{e}_\alpha\) belongs to \(C^{\infty}_c\) and thus the second term belongs to \(E\). We proceed to prove that if \(f\) belongs to \(F\), then \(S\circ V_p \circ Q_p(f)\) belongs to \(E\). Since \(S(F)\subset E\), this amounts to proving that \(V_p\circ Q_p(F) \subset F\). 

Let \(f\in F\). It follows from the definition of \(T_a\) and the second assumption that \(T_a(C^{\infty}_c)\subset F\). Since by construction,   \(Z\) is contained in \(T_a(C^{\infty}_c)\), we deduce that \(Z\) is a subset of \(F\) and thus \(\zeta_p(f)\in F\). Using the decomposition \(f=Q_p(f)+ \zeta_p(f)\) and the fact that \(f\in F\), one gets that \(Q_p(f)\in F\).

Let \(u= V_p \circ Q_p(f)\). Then \((I+K)(u)=Q_p(f)\in F\). We now rely on \eqref{eq_univ_prop} to get that \(u\in F\); that is  \(V_p\circ Q_p(f)\in F\). The proof is complete. 
\end{proof}




\section{Universal property for \(S_a\)}
\label{section:higher regularity}

In this section we deal with higher order Sobolev and H\"older spaces.
As explained in the introduction, a  right inverse \(\overline{S_0}\) to the divergence operator 
\(T_0=\diw\) which is universal in the scales of these spaces, is only available in the literature when \(\Omega\) is at least \(C^2\), see Theorem \ref{theorem:a is zero higher regularity}. Hence, we assume this regularity property on \(\Omega\) throughout this section. We can repeat the same construction as in Section 2 with \(\overline{S_0}\) instead \(S_0\). More precisely, in the definition \eqref{eq_def_S} of \(S\), the operator \(S_0\) has to be replaced by \(\overline{S_0}\). We thus obtain a linear map \(\overline{S_a}\) which  is a right inverse to \(T_a\) and is universal in the scale of Lebesgue spaces.  

We proceed to prove that \(\overline{S_a}\) is also universal in the scale of higher order Sobolev spaces and H\"older spaces. We rely on Lemma \ref{lm_univ_prop} and we first check that the assumption \ref{eq_univ_prop} is satisfied in our framework.

\begin{lemma}\label{lm-univ-sobolev}
Let \(m\geq 0\), \(r> \frac{n}{m+1}\) and  \(a\in W^{m,r}(\Omega ; \R^n)\). Assume that \(\Omega\) is of class \(C^{m+2}\). Then  for every \(1<p\leq r\) and \(f\in L^{p}(\Omega)\),
\[
(I+K)(f)\in W^{m,p}(\Omega) \;\Longrightarrow\;  f\in  W^{m,p}(\Omega).
\]
\end{lemma}
\begin{proof}
We prove by induction on \(m\in \N\) the following more general result:
If \(\Omega\) is of class \(C^{m+2}\) and \(a\in W^{m,r}(\Omega;\re^n)\) with \(r>\frac{n}{m+1}\), then for every \(1<p\leq s\leq r\) and every \(f\in L^{p}(\Omega)\), 
\[
(I+K)(f)\in W^{m,s}(\Omega) \;\Longrightarrow\;  f\in  W^{m,s}(\Omega).
\]
Remember that
$$
(I+K)(f)=\widetilde{f} + \langle a, \overline{S_0}(\widetilde{f})\rangle
 \qquad \textrm{ where } \widetilde{f}=f-\fint_{\Omega}f.
$$
We first consider the case \(m=0\) which is a slight generalization of Lemma \ref{lm_inva_p}. If \(p=s\), there is nothing to prove. Hence, we assume that \(p<s\). If \(s>n\), then we can apply Lemma \ref{lm_inva_p} with \(q=s\), which  gives that \(f\in L^s\). If \(s\leq n\), we rely on the fact that \(\overline{S_0}(\widetilde{f})\in W^{1,p}\) and the Sobolev embedding \(W^{1,p}\subset L^{p^*}\) (observe that we are in the case where \(p<s\leq n\)) to get that \(\overline{S_0}(\widetilde{f})\in L^{p^*}\). This implies that 
\begin{equation}\label{eq794}
\langle a, \overline{S_0}(\widetilde{f}) \rangle \in L^{p_1}, \qquad  \textrm{with } 
\frac{1}{p_1}=\frac{1}{p}-\frac{1}{n}+\frac{1}{r}.
\end{equation}
If \(p_1\geq s\), then \(\widetilde{f}\in L^{s}\), as desired. Otherwise, \(\widetilde{f}\in L^{p_1}\).

We next introduce as in the proof of Lemma \ref{lm_inva_p} the sequence
\[
\frac{1}{p_k} =\frac{1}{p_{k-1}} +\frac{1}{r} -\frac{1}{n}, \qquad k\in \N.
\]
We have already seen that there exists \(k_0\in \N\) such that \(1<p_k<n\) for \(k\leq k_0\) and \(p_{k_0+1}\geq n\). This implies that there exists \(k_1\in \N\) such that \(1<p_k<s\) for every \(k\leq k_1\)  and \(p_{k_1+1}\geq s\). Repeating the  argument leading to \eqref{eq794} for \(p_1, \dots, p_{k_1}\), we finally obtain that \(\langle a, \overline{S_0}(\widetilde{f}) \rangle \in L^{p_{k_1+1}}\). This gives \(\widetilde{f}\in L^{s}\), and completes the proof in the case \(m=0\).

We now assume that the property is true for some \(m\geq 0\) and let us prove it for \(m+1\). Let
$$
  a\in W^{m+1,r}(\Omega ; \R^n)\quad\text{ with }\quad r>\frac{n}{m+2}. 
$$
Let \(1<p\leq s\leq r\) and \(f\in L^{p}(\Omega;\R^n)\) such that
\[
\widetilde{f} + \langle a , \overline{S_0}(\widetilde{f}) \rangle \in W^{m+1,s}
\]
When \(r<n\), we set \(\overline{r}=r^*\) and observe that \(r>\frac{n}{m+2}\) implies \(\overline{r}>\frac{n}{m+1}\). When \(r\geq n\), we take for \(\overline{r}\) any number \(>\frac{n}{m+1}\). In any case, by the Sobolev embeddings, \(a\in W^{m,\overline{r}}\). 
Let
\[
\overline{s}:=\begin{cases}
s^* & \textrm{ if } s<n,\\
\overline{r} & \textrm{ if } s\geq n.
\end{cases}
\]
Then \(\widetilde{f} + \langle a , \overline{S_0}(\widetilde{f}) \rangle \in W^{m,\overline{s}}\). Since \(1<p\leq \overline{s}\leq \overline{r}\),  the induction assumption  yields \(\widetilde{f}\in W^{m, \overline{s}}\). Hence, \(\overline{S_0}(\widetilde{f}) \in W^{m+1, \overline{s}}\). 

When \(s\geq n\), we use that \(a\in W^{m+1, s}\) and rely on  Corollary \ref{corollary:Sob multiply} (i) below (applied with  $s$  instead of \(p\),  $q=\overline{s}$ and with \(m+1\) instead of \(m\)). This gives
$$
  \langle a, \overline{S_0}(\widetilde{f})\rangle\in W^{m+1,s}.
$$
Hence,  \(\widetilde{f}\in W^{m+1,s}\), which proves the induction assumption for \(m+1\).

When \(s<n\), we use that \(\overline{S_0}(\widetilde{f}) \in W^{m+1, s^*}\) and  apply Corollary \ref{corollary:Sob multiply} (ii) (with  $s$  instead of \(p\), \(q=r\) and  \(m+1\) instead of \(m\)). We deduce that
$$
  \langle a, \overline{S_0}(\widetilde{f})\rangle\in W^{m+1,s}
$$
and conclude as before. 
The proof is complete.
\end{proof}

The corresponding statement in the scale of H\"older spaces reads:

\begin{lemma}\label{lm-univ-holder}
 Let \(m\geq 0\), \(\alpha \in (0,1)\) and \(a\in C^{m,\alpha}(\Omegabar)\). Assume that \(\Omega\) is of class \(C^{m+2, \alpha}\).  Then for every \(1< p <\infty\) and for every \(f\in L^{p}(\Omega)\),
\[
(I+K)(f)\in C^{m,\alpha}(\Omegabar) \Longrightarrow f\in C^{m,\alpha}(\Omegabar).
\] 
\end{lemma}

\begin{proof}
Let \(f\in L^{p}\), \(1<p<\infty\), such that for some \(m\geq 0\) and \(\alpha\in (0,1)\), 
\begin{equation}
 \label{eq:bootstrap in Hoelder}
  \widetilde{f} + \langle a ,  \overline{S_0}(\widetilde{f}) \rangle  \in C^{m,\alpha}, 
\end{equation}
where as usual $\widetilde{f}=f-\frac{1}{|\Omega|}\int_{\Omega}f$.
\smallskip

\textit{Step 1.} We shall first show that $\widetilde{f}\in C^{0,\alpha}.$ We can assume that $p<n,$ because if the lemma holds in this case, then it certainly also holds for $p\geq n.$ 
Using that \(C^{m,\alpha}\subset L^{\infty}\),  \( \overline{S_0}(\widetilde{f})\in W^{1,p}\subset L^{p_*}\) and \(a\in L^{\infty}\), we deduce that
\begin{equation}\label{eq847}
  \tilde{f}\in L^{p^*},\quad\text{ with }\quad p^*=\frac{np}{n-p}.
\end{equation}
By a similar argument as in Case 3 in the proof of Lemma \ref{lm_inva_p}, there exists \(k_0\in \N\) such that the sequence \((p_k)_{k\in \N}\) defined by 
\[
p_0=p, \qquad \qquad  \forall k\geq 1, \qquad \frac{1}{p_k}=\frac{1}{p_{k-1}}-\frac{1}{n}=\frac{1}{p}-\frac{k}{n}
\]
satisfies \(1<p_k<n\) for \(k\leq k_0\) and \(p_{k_0+1}\geq n\).
By bootstrapping the argument leading to \eqref{eq847},  we obtain that $\widetilde{f}\in L^{p_{k_0+1}},$  Thus we get
$  \langle a, \overline{S_0}(\widetilde{f})\rangle\in L^{p_{k_0+1}}$ which implies that $\widetilde{f}\in L^{p_{k_0+1}}$.
Hence, $ \overline{S_0}(\widetilde{f})\in W^{1,p_{k_0+1}}\subset W^{1,n}$ which in turn implies by \eqref{eq:bootstrap in Hoelder}  that $\tilde{f}\in L^r$ for any $r<\infty.$ It follows that \( \overline{S_0}(\widetilde{f})\)  belongs to \(\cap_{1<r<\infty}W^{1,r}\). By the Morrey embedding, we deduce therefrom that \( \overline{S_0}(\widetilde{f})\)  belongs to \(\cap_{0<\beta<1}C^{0,\beta}\). Since \(a\in C^{0,\alpha}\), this implies  that \(\langle a,  \overline{S_0}(\widetilde{f})\rangle\) belongs to \(C^{0,\alpha}\). Therefore, by  \eqref{eq:bootstrap in Hoelder}, \(\widetilde{f}\in C^{0,\alpha}\).
\smallskip

\textit{Step 2.} If $m=0$ we are done by Step 1.
Otherwise, for every \(0\leq k \leq m-1\), the operator \(u\mapsto \langle a , u \rangle\) maps continuously \(C^{k+1,\alpha}\) into \(C^{k+1,\alpha}\), see  \cite[Theorem 16.28]{Csato-Dacorogna-Kneuss}.  
Together with the facts that   \( \overline{S_0} (C^{k,\alpha}) \subset C^{k+1,\alpha}\) and \(C^{m,\alpha} \subset C^{k+1,\alpha}\), it follows by induction on \(k=0, \dots, m-1\) that  \(\widetilde{f}\in C^{m,\alpha}(\Omegabar)\). 
\end{proof}

\smallskip

\begin{proof}[Proof of Theorem \ref{th-universal-property}]
In the scale of Sobolev spaces, we apply Lemma \ref{lm_univ_prop} with $\overline{S_a}$ instead of \(S_a\), to the sets \(E=W^{m+1,p}_z(\Omega;\R^n)\) and \(F=W^{m,p}(\Omega)\). In view of Lemma \ref{lm-univ-sobolev}, the assumption \eqref{eq_univ_prop} is satisfied. Theorem \ref{theorem:a is zero higher regularity} implies that \(\overline{S} : f\mapsto \overline{S_0}\left(f-\frac{1}{|\Omega|}\int_{\Omega}f\right)\) maps continuously \(W^{m,p}\) into \(W^{m+1,p}_{z}\).
The conclusion follows in that case.

In the scale of H\"older spaces, we rely on Lemma \ref{lm_univ_prop}  with \(E=C^{m+1,\alpha}_z(\Omegabar;\R^n)\) and \(F=C^{m,\alpha}(\Omegabar)\). Here, assumption \eqref{eq_univ_prop} follows from Lemma \ref{lm-univ-holder}. Theorem \ref{theorem:a is zero higher regularity} now implies that \(\overline{S}\) maps continuously \(C^{m,\alpha}\) into \(C^{m+1,\alpha}_{z}\). The proof is complete.
\end{proof}

\section{Non-existence results in $L^1$ and $L^{\infty}$}

First we discuss the existence in $L^1,$ which is easier.
These results have nothing to do with the boundary conditions, and the proofs are almost identical to the case $a=0,$ see for instance \cite{Bourgain-Brezis}.

\begin{theorem}
\label{theorem:non reg in L1}
Let $\Omega\subset\re^n$ be a bounded Lipschitz set. Let $a\in L^{\infty}(\Omega;\re^n).$
Then there exists $f\in L^1(\Omega)$ such that there is no $u\in W^{1,1}(\Omega;\re^n)$ with $\diw u+\langle a,u\rangle =f.$
\end{theorem}

\begin{proof} 
Suppose by contradiction that for all $f\in L^1$ there exists a $u\in W^{1,1}$ such that
$$
  T_a(u)=\diw u+\langle a;u\rangle=f.
$$
Hence the map $T_a:W^{1,1}\to L^1$ is onto. Moreover it is bounded and linear. Hence by the open mapping theorem there is a constant $C>0,$ such that for every $f\in L^1$ there exists a $u$ with $T_a(u)=f$ and
$$
  \|u\|_{W^{1,1}}\leq C_1 \|f\|_{L^1}.
$$
We will show that for all $\varphi\in W_0^{1,n}$ and all $f\in L^1$ we have
\begin{equation}
 \label{eq:proof:first contradiction f and u L1}
  \left|\intomega \varphi f\right|\leq C(\Omega,\|a\|_{L^{\infty}},C_1) \|\varphi\|_{W^{1,n}}\|f\|_{L^1}.
\end{equation}
Indeed we have that
\begin{align*}
 \left|\intomega \varphi f\right|=&\left|\intomega \varphi(\diw u+\langle a;u\rangle)\right|=\left|-\intomega (\nabla \varphi)\, u+\intomega \varphi\langle a;u\rangle\right| 
 \smallskip \\
 \leq& \|\nabla \varphi\|_{L^n}\|u\|_{L^{\frac{n}{n-1}}}+\|a\|_{L^{\infty}}\|\varphi\|_{L^n}\|u\|_{L^{\frac{n}{n-1}}}.
\end{align*}
We now use the continuous embedding $W^{1,1}\subset L^{\frac{n}{n-1}}$ and obtain that
\begin{align*}
 \left|\intomega \varphi f\right|\leq C \|\varphi\|_{W^{1,n}}\|u\|_{W^{1,1}}\leq C \|\varphi\|_{W^{1,n}}\|f\|_{L^1}.
\end{align*}
This proves \eqref{eq:proof:first contradiction f and u L1}. We are now able to conclude. Consider the map $M_{\varphi}:L^1\to \re$, for every fixed $\varphi\in W^{1,n}_0$, given by $M_{\varphi}(f)=\int \varphi f.$ Then by \eqref{eq:proof:first contradiction f and u L1} we have that
$$
  M_\varphi\in (L^1)^{\ast}=L^{\infty}
\quad \textrm{ and } \quad 
  \|\varphi\|_{L^{\infty}}\leq C\|\varphi\|_{W^{1,n}}.
$$
This is a contradiction to the non embedding \(W^{1,n}\not\subset L^{\infty}\).
\end{proof}
\smallskip

We now deal with the
 $L^{\infty}$ non-existence.

\begin{theorem} 
Let $\Omega\subset\re^n$ be a bounded open set. Let $a\in W^{1,\infty}(\Omega;\re^n).$
Then there exists $f\in L^{\infty}(\Omega)$ such that there is no $u\in W^{1,\infty}(\Omega)$ with $\diw u+\langle a;u\rangle =f.$
\end{theorem}

\begin{proof} For simplicity we present the proof for $n=2.$ We assume by contradiction that for every $f\in L^{\infty}$ there exists $u$ satisfying (as in the proof of Theorem \ref{theorem:non reg in L1})
$$
  T_a(u)=\diw u+\langle a;u\rangle =f,\quad \|u\|_{W^{1,\infty}}\leq C\|f\|_{L^{\infty}}.
$$
Let $\psi\in C_c^{\infty}(\Omega)$ and define 
 $f\in L^{\infty}(\Omega)$ as ($\psi_{x_i},\psi_{x_i x_j}$ denote the partial derivatives of \(\psi\))
$$
  f=\operatorname{sign}\psi_{x_1x_2}.
$$
Let $u$ be a solution of $T_a(u)=f$ with $\|u\|_{W^{1,\infty}}\leq C\|f\|_{L^{\infty}}=C.$
We therefore obtain that
\begin{align*}
 \|\psi_{x_1x_2}\|_{L^1}=&\left|\int\psi_{x_1x_2}f\right|=\left|\int\psi_{x_1x_2}\left(\diw u+\langle a,u\rangle\right) \right|
 \smallskip \\
 =&\left|\int\left(\psi_{x_1x_1}\frac{\partial u_1}{\partial x_2}+\psi_{x_2x_2}\frac{\partial u_2}{\partial x_1}-\psi_{x_1}\frac{\partial }{\partial x_2}\langle a,u\rangle\right)\right|
 \smallskip \\
 \leq & \left(\|\psi_{x_1x_1}\|_{L^1}+\|\psi_{x_2x_2}\|_{L^1}\right)\|u\|_{W^{1,\infty}}+\|\psi_{x_1}\|_{L^1}\|a\|_{W^{1,\infty}}\|u\|_{W^{1,\infty}}.
\end{align*}
We now use that $\|\psi_{x_1}\|_{L^1}\leq C \|\psi_{x_1x_1}\|_{L^1}$ for some constant $C>0$ depending only on $\Omega.$ This gives
$$
  \|\psi_{x_1x_2}\|_{L^1}\leq C(1+\|a\|_{W^{1,\infty}})\left(\|\psi_{x_1x_1}\|_{L^1}+\|\psi_{x_2x_2}\|_{L^1}\right),
$$
which is a contradiction to the non-inequality of Ornstein \cite{Ornstein}.
\end{proof}

\section{Appendix}

\subsection{Multiplication in Sobolev spaces}

\begin{lemma}\label{lemma_multiplier}
Let \(1\leq p\leq q \leq \infty\) and  let $l\in \N$ be an integer such that $n-pl>0$ and let \(m\in \N^*\). Define
$$
  p(l)=\frac{np}{n-pl}
$$
We assume that 
\begin{equation}
 \label{eq:assumption on pml}
  q> \frac{n}{m+l}.
\end{equation}
Then for every \(f\in W^{m,p(l)}(\Omega)\) and \(g\in W^{m,q}(\Omega)\), the function \(f g \) belongs to \(W^{m,p}(\Omega)\) and 
\[
\|f g\|_{W^{m,p}} \leq C \|f\|_{W^{m,p(l)}} \|g\|_{W^{m,q}}
\]
for some constant \(C>0\) which depends only on \(n,m,p,q\) and \(l\). 
\end{lemma}

We will only use the cases $l=0$ and $l=1,$ which we emphasize as a Corollary.

\begin{corollary}
\label{corollary:Sob multiply}
Let \(1\leq p\leq q \leq \infty\) and  \(m\in \N^*\). 

(i) We assume that  \(q> \frac{n}{m}\).
Then for every \(f\in W^{m,p}(\Omega)\) and \(g\in W^{m,q}(\Omega)\), the function \(f g \) belongs to \(W^{m,p}(\Omega)\) and 
\[
\|f g\|_{W^{m,p}} \leq C \|f\|_{W^{m,p}} \|g\|_{W^{m,q}}.
\]

(ii)  We assume that  \(q> \frac{n}{m+1}\), $p<n$ and $p^*$ denotes the Sobolev conjugate of $p.$
Then for every \(f\in W^{m,p^*}(\Omega)\) and \(g\in W^{m,q}(\Omega)\), the function \(f g \) belongs to \(W^{m,p}(\Omega)\) and 
\[
\|f g\|_{W^{m,p}} \leq C \|f\|_{W^{m,p^*}} \|g\|_{W^{m,q}}.
\]
\end{corollary}

\begin{proof}[Proof of Lemma \ref{lemma_multiplier}.]
Let \(f\in W^{m,p(l)}\) and \(g\in W^{m,q}\). By the Sobolev and the Morrey embeddings, for every \(k=0, \dots, m\), \(W^{k,p(l)} \subset L^{p_{k}(l)}\) with
\begin{equation}
 \label{eq:qk}
  \frac{1}{p_k(l)}:=\begin{cases}
  \frac{1}{p(l)} - \frac{k}{n}=\frac{1}{p} - \frac{k+l}{n} & \textrm{ if } k<\frac{n}{p(l)}\quad\Leftrightarrow\quad k+l<\frac{n}{p},\\
  \text{as small as we wish} & \textrm{ if } k=\frac{n}{p(l)},\\
  0 & \textrm{ if } k>\frac{n}{p(l)},
  \end{cases}
\end{equation}
where the third case has to be understood as $p_k(l)=\infty.$
The same holds true for \(W^{k,q}\subset L^{q_k}(\Omega)\) with
\begin{equation}
 \label{eq:pk}
  \frac{1}{q_k}:=\begin{cases}
  \frac{1}{q} - \frac{k}{n} & \textrm{ if } k<\frac{n}{q},\\
  \text{as small as we wish}& \textrm{ if } k=\frac{n}{q},\\
  0 & \textrm{ if } k>\frac{n}{q}.
  \end{cases}
\end{equation}
By considering each case successively, one can check that  for every \(k=0,\dots,m\),
\begin{equation}\label{eq-1056}
  \frac{1}{p_k(l)} + \frac{1}{q_{m-k}} \leq \frac{1}{p}.
\end{equation}
Indeed,  if we are in the first cases of both \eqref{eq:qk} and \eqref{eq:pk}, then using that \(q\geq \frac{n}{m+l}\), one gets
\begin{align*}
 \frac{1}{p_k(l)}+\frac{1}{q_{m-k}}=\frac{1}{p}+\frac{1}{q}-\frac{m+l}{n}\leq \frac{1}{p}.
\end{align*}

Assume next that \(k=\frac{n}{p(l)}\) and \(m-k<\frac{n}{q}\). Then \eqref{eq-1056} is satisfied provided that \(\frac{1}{q_{m-k}}<\frac{1}{p}\), or equivalently, 
\begin{equation}\label{eq-1064}
\frac{1}{q}-\frac{m-k}{n}<\frac{1}{p}.
\end{equation}
This holds true when \(q>p\). When \(q=p\), we have
\[
k=\frac{n}{p(l)}=\frac{n}{p}-l = \frac{n}{q}-l.
\]
Since by assumption \(q>\frac{n}{m+l}\), one has \(k<m\), which implies \eqref{eq-1064}. 

In the case \(k<\frac{n}{p(l)}\) and \(m-k=\frac{n}{q}\), the proof is essentially the same (when \(l=0\), we  use that \(q>\frac{n}{m}\) which implies that \(k>0\)). The remaining cases are obvious and we omit them. This completes the proof of \eqref{eq-1056}.
 
One also verifies that
$$
  \quad  \frac{1}{p_m(l)} + \frac{1}{q_m}\leq \frac{1}{p}.
$$
In particular, since \(f\in L^{p_m(l)}\), \(g\in L^{q_m}\), the H\"older inequality implies \(fg\in L^{p}\). 

We now prove that \(D^m(fg)\in L^{p}\), where \(D^m\) denotes the tensor of all partial derivatives of order \(m\). By the Leibniz rule,
\begin{equation*}
  \|D^m(fg)\|_{L^{p}}\leq C \sum_{k=0}^{m}\||D^{m-k}f||D^k g|\|_{L^{p}}.
\end{equation*}
Since \(D^{m-k}f\in W^{k,p(l)}\subset L^{p_k(l)}\) and \(D^k g \in W^{m-k,q}\subset L^{q_{m-k}}\), the H\"older inequality and \eqref{eq-1056} imply that each term of the above sum belongs to \(L^{p}\) and
\[
\|D^m(fg)\|_{L^{p}}\leq C' \sum_{k=0}^{m}\|D^{m-k}f\|_{L^{p_k(l)}}\|D^k g\|_{L^{q_{m-k}}}\leq C''\|f\|_{W^{m,p(l)}}\|g\|_{W^{m,q}}.
\] 
This completes the proof.
\end{proof}

\subsection{A density lemma}

\begin{lemma}\label{lemma_complementing_density}
Let \(X\) be  a Banach space, \(F\) a closed subspace of \(X\) of codimension \(k\geq 1\) and \(D\) a dense subset of \(X\). Then there exists a complemented subspace \(G\) for \(F\)  such that \(G\) has a basis contained in \(D\).
\end{lemma}

\begin{proof}
Let \(\widetilde{G}\) be a complemented subspace of \(F\) and \((e_1, \dots, e_k)\) be a basis of \(\widetilde{G}\). We consider an approximation \((e_{1}^{n},\dots, e_{k}^{n})_{n\in\N}\subset D^{k}\) of this basis: 
\begin{equation}\label{eq413}
\forall i\in \{1, \dots, k\}, \qquad \lim_{n\to +\infty}e_{i}^{n}=e_i
\end{equation}
and we define 
\[
G^n:=\textrm{Vect}\,(e_{1}^{n},\dots, e_{k}^{n}).
\]
We claim that there exists an \(n_0\in \N\) such that (actually this is true for all $n$ sufficiently large), 
\begin{equation}\label{eq:lambda zero}
  \forall \lambda=(\lambda_1, \dots, \lambda_k) \in \R^k,\qquad
  \sum_{j=1}^k\lambda_j e_{j}^{n_0} \in F \Longrightarrow \lambda=0.
\end{equation}
Indeed, assume by contradiction that  for every \(n\in \N\), there exists a \(\lambda^{n}=(\lambda_{1}^{n}, \dots, \lambda_{k}^{n})\in \R^k \setminus \{0\}\) such that 
\[
\sum_{j=1}^{k}\lambda_{j}^{n}e_{j}^{n}\in F.
\]
We can assume that \(\|\lambda^{n}\|=1\) for some fixed norm \(\|\cdot\|\) in \(\R^k\). Hence, up to a subsequence (we do not relabel), \((\lambda^{n})_{n\in \N}\) converges to some \(\lambda \in \R^k\) such that \(\|\lambda\|=1\). Then 
\[
\lim_{n\to +\infty}\sum_{j=1}^{k}\lambda_{j}^{n}e_{j}^{n} = \sum_{j=1}^{k}\lambda_{j}e_{j}.
\]
Since \(F\) is closed, this implies that \(\sum_{j=1}^{k}\lambda_{j}e_{j}\in F\). Using that \(\widetilde{G}\cap F =\{0\}\), we deduce that \(\sum_{j=1}^{k}\lambda_{j}e_{j}=0\), which implies that \(\lambda=0\) in view of the fact that \((e_j)_{1\leq j \leq k}\) is a basis of $\widetilde{G}.$ This contradicts the fact that \(\|\lambda\|=1\) and the claim \eqref{eq:lambda zero} is proved.
\smallskip

Set $G=G^{n_0}$ which satisfies by \eqref{eq:lambda zero}, $G\cap F=\{0\}$. Since \(F\) has codimension \(k\), this proves that \(X=F\oplus G\), as desired.

\end{proof}

\smallskip

\noindent\textbf{Acknowledgements} The second author was supported by Chilean Fondecyt Iniciaci\'on grant nr. 11150017 and  is member of the Barcelona Graduate School of 
Mathematics. He also acknowledges financial support from the Spanish Ministry of Economy and Competitiveness, through the ``Mar\'ia de Maeztu'' Programme for Units of Excellence in RD (MDM-2014-0445).


\end{document}